\documentclass[12pt,reqno]{amsart}
\usepackage{amssymb,amsthm,url}
\usepackage[breaklinks=true]{hyperref}
\usepackage[utf8]{inputenc}
\usepackage[OT2,T1]{fontenc}
\usepackage{bbm}
\usepackage{pdfsync}
\usepackage{ytableau}
\usepackage{graphicx}
\swapnumbers
\let\<\langle
\let\>\rangle

\advance\hoffset by -0.5truein
\DeclareSymbolFont{cyrletters}{OT2}{wncyr}{m}{n}
\DeclareMathSymbol{\Sha}{\mathalpha}{cyrletters}{"58}

\theoremstyle{definition}
\newtheorem{definition}{Definition}[section]

\newtheorem{example}[definition]{Example}

\theoremstyle{plain}

\newtheorem{theorem}[definition]{Theorem}

\def\Lie{\mathrm{Lie}}

\def\Com{\mathrm{Com}}

\def\M{\mathrm{M}}

\def\Com{\mathrm{Com}}

\def\TP{\mathrm{TP}}
\def\F{\mathrm{F}}
\def\GD{\mathrm{GD}}
\def\manifold{\mathrm{manifold}}

\def\Com{\mathrm{Com}}

\usepackage[all]{xy}
\usepackage{pdfsync}
\usepackage{ytableau}

\newcommand{\sudda}[1]{}

\begin{document}

\title[On the free metabelian transposed Poisson and F-manifold algebras]{On the free metabelian transposed Poisson and F-manifold algebras}

\author{K. Abdukhalikov}

\address{UAE University, Al Ain, UAE}

\email{abdukhalik@uaeu.ac.ae}

\author{B. K. Sartayev*}

\address{Narxoz University, Almaty, Kazakhstan and UAE University, Al Ain, UAE}

\email{baurjai@gmail.com}

\keywords{Transposed Poisson algebra, F-manifold algebra, metabelian identity, free algebra, polynomial identities}

\subjclass[2020]{17A30, 17A50, 17D25}

\thanks{${}^{*}$Corresponding author: Bauyrzhan Sartayev   (baurjai@gmail.com)}

\maketitle

\begin{abstract} 
In this paper, we consider free transposed Poisson algebra and free F-manifold algebra with an additional metabelian identity. We construct a linear basis for both free metabelian transposed Poisson algebra and free metabelian F-manifold algebra.
\end{abstract}

\section{Introduction}

Algebras with metabelian identities give significant results in classical problems of algebra such as Specht problem, Gr\"obner-Shirshov theory, describing of symmetric polynomials etc. Another application of algebras with additional metabelian identity is to solve the embedding problem as embedding of metabelian Malcev algebras into corresponding alternative algebras \cite{MMal't}. Another result in this direction is the embedding of metabelian Lie algebras into corresponding associative algebras \cite{MS2022}. The variety of metabelian Lie algebras plays a big role in various directions of mathematics \cite{CHF, Romankov}.

Let us focus on well-known results on the bases of the free metabelian algebras with an additional identity of degree $3$. Foundational work on the basis of the free metabelian Lie algebras was given in \cite{Bahturin1973}. A basis of the free metabelian Leibniz algebra was constructed in \cite{MLeib}. Bases of the free metabelian Novikov and Lie-admissible algebras were given in \cite{DAS}. The construction of the basis for the free metabelian Malcev algebra was shown in \cite{MMal't}.

In this paper, we consider free transposed Poisson and $\F$-manifold algebras with additional metabelian identities. Although it may seem that these two varieties of algebras do not have a common ground for considering them, the variety of transposed Poisson algebras lies in the variety of $\F$-manifold algebras, and in addition, it lies in the variety of $\GD$-algebras \cite{Com-GD}. The identities of GD-algebra emerged in the paper \cite{GD83} as a tool for constructing Hamiltonian operators in formal calculus of variations. It is well-known that GD-algebras are in one-to-one correspondence with quadratic Lie conformal algebras playing an important role in the theory of vertex operators \cite{Xu2000}. Constructing the basis of the free $\GD$-algebra is still an open problem. The analogical problem is not solved for the free transposed Poisson algebra. At the moment, even the basis of a free one-generated transposed Poisson algebra is unknown. The same problem for metabelian $\GD$-algebra is not trivial. However, the basis of $\F$-manifold operad via pre-Lie operad was constructed in \cite{Dotsenko2019}. Also, the basis of the free special $\GD$-algebra is given in \cite{GS2023}. 

Many open problems associated with $\GD$-algebras and their generalizations are stated in \cite{TP3}. One of these problems is to prove that any transposed Poisson algebra is special, i.e., it can be embedded into appropriate differential Poisson algebra. Using this conjecture, the basis of the free special $\GD$-algebra and the compatible condition of the transposed Poisson algebra, it may be possible to find the basis of a free transposed Poisson algebra.

The listed algebras have attracted much attention in recent years. For example, the results on the transposed Poisson algebras were given in \cite{TP1, TP2, TP3, TP4}. The application of $\F$-manifold algebras and operad was shown in \cite{F-man1, F-man2, F-man3, F-man4}. Based on the above papers and open problems, we believe that there is enough motivation to consider transposed Poisson and $\F$-manifold algebras with the additional metabelian identity.

In this paper, all algebras are defined over a field of characteristic~0.

\section{Free metabelian transposed Poisson algebra}

An algebra with two binary operations $\cdot$ and $[\cdot,\cdot]$ is called transposed Poisson if the first one is an operation of associative-commutative algebra and the second one is an operation of Lie algebra with the following additional identity:
\begin{equation}\label{tp}
a[b,c]=1/2([ab,c]+[b,ac]).
\end{equation}
Metabelian identity  for transposed Poisson algebra is defined as follows:
\begin{equation}\label{10}
[[a,b],[c,d]]=[[a,b],cd]=[ab,cd]=abcd=
ab[c,d]=[a,b][c,d]=0.
\end{equation}
As a consequence, we obtain
\begin{equation}\label{11}
[[[a,b],c],d]=[[[a,b],d],c]    
\end{equation}
and
\begin{equation}\label{12}
[[ab,c],d]=[[ab,d],c].
\end{equation}
In addition, one can obtain that
\begin{equation}\label{13}
[[ab,c],d]-[[bc,a],d]+[[cd,a],b]-[[ad,c],b]=0
\end{equation}
and
\begin{equation}\label{14}
[abc,d]=[abd,c]
\end{equation}
hold in metabelian transposed Poisson algebra.

Let $X=\{x_1,x_2\ldots\}$ be a countable set of generators. We denote by $\TP\<X\>$ and $\M\TP\<X\>$ the free transposed Poisson algebra and free metabelian transposed Poisson, respectively.

It is easy to see that the identity (\ref{tp}) provides the spanning monomials of $\TP\<X\>$ algebra. More precisely, the compatible condition of transposed Poisson algebra provides the following rewriting rule:
\begin{equation}\label{tprule}
\Com(\Lie)\rightsquigarrow\Lie(\Com),
\end{equation}
i.e., every monomial of $\TP\<X\>$ algebra can be written as a sum of monomials of the following form:
$$[[\cdots[\mathcal{C}_n,\mathcal{C}_{n-1}]\cdots,\mathcal{C}_2],\mathcal{C}_1],$$
where $\mathcal{C}_i$ is a pure monomial of the free associative-commutative algebra.

It is easy to verify that the free transposed Poisson algebra does not satisfy distribute law, in the sense of \cite{Markl}. For this reason, the transposed Poisson operad is not Koszul. This result was proved in \cite{Dzhuma}.

Let us consider $3$ types of monomials in $\M\TP\<X\>$ algebra:
$$[[\cdots[[[x_{i_1},x_{i_2}],x_{i_3}],x_{i_4}],\cdots],x_{i_n}],$$
$$[[\cdots[[x_{i_1}x_{i_2},x_{i_3}],x_{i_4}],\cdots],x_{i_n}]$$
and
$$[[\cdots[x_{i_1}x_{i_2}x_{i_3},x_{i_4}],\cdots],x_{i_n}],$$
where $n\geq 4$. For each type of monomials, we define the sets $\mathcal{T}_{n,1}$, $\mathcal{T}_{n,2}$ and $\mathcal{T}_{n,3}$ as follows:
$$\mathcal{T}_{n,1}=\{[[\cdots[[[x_{i_1},x_{i_2}],x_{i_3}],x_{i_4}],\cdots],x_{i_n}] \;|\; i_1>i_2\leq i_3 \leq i_4 \leq\cdots \leq i_n\},$$
\begin{gather*}
\mathcal{T}_{n,2}=\{[[\cdots[[x_{i_1}x_{i_2},x_{i_3}],x_{i_4}],\cdots],x_{i_n}] \;|\; i_1 \leq i_2,\; i_1<i_3\leq\cdots\leq i_n \;\textrm{or}\; \\
i_3\leq\cdots\leq i_n<i_1 \leq i_2\},    
\end{gather*}
$$\mathcal{T}_{n,3}=\{[[\cdots[x_{i_1}x_{i_2}x_{i_3},x_{i_4}],\cdots],x_{i_n}] \;|\; i_1\leq i_2\leq i_3 \leq i_4 \leq\cdots \leq i_n\ \}.$$
We set $\mathcal{T}_n=\mathcal{T}_{n,1}\cup\mathcal{T}_{n,2}\cup\mathcal{T}_{n,3}$. For $n=1,2$ and $3$, we define 
$$\mathcal{T}_1=\{x_i\},\;\mathcal{T}_2=\{[x_{i_1},x_{i_2}],x_{j_1}x_{j_2} \;|\; i_1<i_2,\; j_1\leq j_2\},$$
$$\mathcal{T}_3=\{x_{i_1}x_{i_2}x_{i_3}, [[x_{j_1},x_{j_2}],x_{j_3}], [x_{i_1}x_{i_2},x_{k_3}] \;|\; i_1\leq i_2\leq i_3,\;j_1<j_2,\;j_1<j_3 \}.$$

\begin{theorem}
The set $\bigcup_i \mathcal{T}_i$ is a linear basis of the algebra $\M\TP\<X\>$.
\end{theorem}
\begin{proof}
The result is obvious for degrees $1$, $2$ and $3$. For monomials of greater degrees, we first show that any monomial of $\M\TP\<X\>$ algebra can be written as a sum of monomials from $\bigcup_i\mathcal{T}_i$, where $i\geq 4$. The rewriting rule (\ref{tprule}) provides spanning set 
$$[[\cdots[\mathcal{C}_n,\mathcal{C}_{n-1}]\cdots,\mathcal{C}_2],\mathcal{C}_1].$$
Metabelian identities provide that the degree of $\mathcal{C}_n$ not greater than $3$, and  $\mathcal{C}_{i}$ is a gene\-rator, where $i=1,\ldots,n-1$. Otherwise, this monomial is equal to $0$.
For $\mathcal{C}_n$, we consider $3$ cases.

Case 1: Degree of $\mathcal{C}_n$ is $1$. In this case $\mathcal{C}_n$ is a generator, i.e., we have $$[[\cdots[[[x_{i_1},x_{i_2}],x_{i_3}],x_{i_4}],\cdots],x_{i_n}],$$
which is a pure Lie word. In this case, the spanning set is $\mathcal{T}_{n,1}$ by \cite{Bahturin1973}.

Case 2: Degree of $\mathcal{C}_n$ is $2$. In this case, we consider the monomials
$$[[\cdots[[x_{i_1}x_{i_2},x_{i_3}],x_{i_4}],\cdots],x_{i_n}].$$
The generators $x_{i_3},x_{i_4},\ldots,x_{i_n}$ are rearrangeable by (\ref{12}). The conditions
$i_1 \leq i_2,\; i_1<i_3\leq\cdots\leq i_n$ or
$i_3\leq\cdots\leq i_n<i_1 \leq i_2$ can be achieved by (\ref{13}) as follows:
\noindent for given arbitrary monomial
\begin{multline*}
[[\cdots[[x_{j_1}x_{j_2},x_{j_3}],x_{j_4}],\cdots],x_{j_n}]=[[\cdots[[x_{j_1}x_{j_2},x_{j_t}],x_{j_r}],\cdots],x_{j_n}]=\\
-[[\cdots[[x_{j_t}x_{j_r},x_{j_1}],x_{j_2}],\cdots],x_{j_n}]+[[\cdots[[x_{j_t}x_{j_1},x_{j_r}],x_{j_2}],\cdots],x_{j_n}]+\\
[[\cdots[[x_{j_r}x_{j_2},x_{j_t}],x_{j_1}],\cdots],x_{j_n}],
\end{multline*} 
where $j_t$ and $j_r$ are minimal and maximal indexes, respectively. In the obtained result, the first and second monomials satisfy the given conditions. For the third monomial, we have
\begin{multline*}
[[\cdots[[x_{j_r}x_{j_2},x_{j_t}],x_{j_1}],\cdots],x_{j_n}]=[[\cdots[[x_{j_r}x_{j_2},x_{j_t}],x_{j_{r-1}}],\cdots],x_{j_n}]=\\
[[\cdots[[x_{j_2}x_{j_t},x_{j_r}],x_{j_{r-1}}],\cdots],x_{j_n}]-[[\cdots[[x_{j_t}x_{j_{r-1}},x_{j_r}],x_{j_{2}}],\cdots],x_{j_n}]+\\
[[\cdots[[x_{j_{r-1}}x_{j_r},x_{j_t}],x_{j_{2}}],\cdots],x_{j_n}].
\end{multline*}
All obtained monomials satisfy the given conditions.
 
Case 3: Degree of $\mathcal{C}_n$ is $3$. In this case, we consider the monomials
$$[[\cdots[x_{i_1}x_{i_2}x_{i_3},x_{i_4}],\cdots],x_{i_n}].$$
The generators $x_{i_4},\ldots,x_{i_n}$ are rearrangeable by (\ref{12}). By (\ref{14}), $x_{i_3}$ and $x_{i_4}$ can be ordered.

To prove that the set $\bigcup_i \mathcal{T}_i$ is the basis of $\M\TP\<X\>$ we construct a multiplication table for algebra $A\<X\>$ with the basis $\bigcup_i \mathcal{T}_i$. Here $A\<X\>$ is a algebra with two multiplications $\circ$ and $\bullet$, and we define the multiplications as follows:
\[
\begin{gathered}
X_1*X_2=0\;\; \text{if $X_1,X_2\in\bigcup_i \mathcal{T}_i$, $\mathrm{deg}(X_1),\mathrm{deg}(X_2)>1$ and $*$ is $\bullet$ or $\circ$};
\end{gathered}
\]
To save space instead of $\bullet$ and $\circ$ we write $[\cdot,\cdot]$ and $\cdot$, respectively. Firstly, we define the multiplication table from degree $5$.
We define the multiplication of basis monomials with generator $x_t$ under the operation $\bullet$.

For $$[[\cdots[[[x_{i_1},x_{i_2}],x_{i_3}],x_{i_4}],\cdots],x_{i_n}]\bullet x_t,$$
the multiplication is defined according to the multiplication table of free metabelian Lie algebras. For 
$[[\cdots[[x_{i_1}x_{i_2},x_{i_3}],x_{i_4}],\cdots],x_{i_n}]\bullet x_t,$ we consider $2$ cases:
\begin{multline*}
[[\cdots[[x_{i_1}x_{i_2},x_{i_3}],x_{i_4}],\cdots],x_{i_n}]\bullet x_t=
[[\cdots[[x_tx_{i_2},x_{i_1}],x_{i_3}],\cdots],x_{i_n}]-\\
[[\cdots[[x_{t}x_{i_n},x_{i_1}],x_{i_2}],\cdots],x_{i_{n-1}}]+[[\cdots[[x_{t}x_{i_1},x_{i_2}],x_{i_3}],\cdots],x_{i_n}]-\\
[[\cdots[[x_{t}x_{i_{n-1}},x_{i_1}],x_{i_2}],\cdots],x_{i_n}]+[[\cdots[[x_{i_{n-1}}x_{i_n},x_{t}],x_{i_1}],\cdots],x_{i_{n-2}}],
\end{multline*}
where $t$ is a minimal index;
\begin{multline*}
[[\cdots[[x_{i_1}x_{i_2},x_{i_3}],x_{i_4}],\cdots],x_{i_n}]\bullet x_t=\\
[\cdots[[[[\cdots[[x_{i_1}x_{i_2},x_{i_3}],x_{i_4}],\cdots],x_{i_m}],x_{t}],x_{i_m+1}]\cdots,x_{i_n}],
\end{multline*}
where $i_m\leq t\leq i_{m+1}$;

For
$[[\cdots[x_{i_1}x_{i_2}x_{i_3},x_{i_4}],\cdots],x_{i_n}]\bullet x_t,$
we consider $2$ cases:
\begin{multline*}
[[\cdots[x_{i_1}x_{i_2}x_{i_3},x_{i_4}],\cdots],x_{i_n}]\bullet x_t=\\
[\cdots[[[[\cdots[x_{i_1}x_{i_2}x_{i_3},x_{i_4}],\cdots],x_{i_m}],x_{t}],x_{i_m+1}]\cdots,x_{i_n}];
\end{multline*}
where $i_m\leq t\leq i_{m+1}$.
$$[[\cdots[x_{i_1}x_{i_2}x_{i_3},x_{i_4}],\cdots],x_{i_n}]\bullet x_t=
[\cdots[[[[x_{i_1}x_{i_2}x_{t},x_{i_3}],x_{i_4}],\cdots,x_{i_n}],$$
where $t\leq i_{3}$;

Now, we define the multiplication of basis monomials with generator $x_t$ under the operation $\circ$. For
$[[\cdots[[[x_{i_1},x_{i_2}],x_{i_3}],x_{i_4}],\cdots],x_{i_n}]\circ x_t,$
we consider $2$ cases:
\begin{gather*}
[[\cdots[[[x_{i_1},x_{i_2}],x_{i_3}],x_{i_4}],\cdots],x_{i_n}]\circ x_t=1/2^{n-1}([[\cdots[[x_{i_2}x_{i_1},x_{i_3}],x_{i_4}],\cdots],x_{i_n}]-\\
[[\cdots[[x_{i_2}x_{i_n},x_{i_1}],x_{i_3}],\cdots],x_{i_{n-1}}]-[[\cdots[[x_{i_2}x_{i_{n-1}},x_{i_n}],x_{i_1}],\cdots],x_{i_{n-2}}]+\\
[[\cdots[[x_{i_{n-1}}x_{i_n},x_{i_2}],x_{i_1}],\cdots],x_{i_{n-2}}])
\end{gather*}
where $t$ is not a minimal index;
\begin{gather*}
[[\cdots[[[x_{i_1},x_{i_2}],x_{i_3}],x_{i_4}],\cdots],x_{i_n}]\circ x_t=1/2^{n-1}([[\cdots[[[x_{i_1}x_t,x_{i_2}],x_{i_3}],x_{i_4}],\cdots],x_{i_n}]-\\
[[\cdots[[[x_{i_2}x_t,x_{i_1}],x_{i_3}],x_{i_4}],\cdots],x_{i_n}]),
\end{gather*}
where $t$ is a minimal index;

For 
$[[\cdots[[x_{i_1}x_{i_2},x_{i_3}],x_{i_4}],\cdots],x_{i_n}]\circ x_t,$
we consider $2$ cases:
\begin{multline*}
[[\cdots[[x_{i_1}x_{i_2},x_{i_3}],x_{i_4}],\cdots],x_{i_n}]\circ x_t=\\
1/2^{n-2}([\cdots[[[[\cdots[x_{i_1}x_{i_2}x_{i_3},x_{i_4}],\cdots],x_{i_m}],x_{t}],x_{i_m+1}]\cdots,x_{i_n}]),
\end{multline*}
where $i_m\leq t\leq i_{m+1}$;
$$[[\cdots[[x_{i_1}x_{i_2},x_{i_3}],x_{i_4}],\cdots],x_{i_n}]\circ x_t=1/2^{n-2}(
[\cdots[[[[x_{i_1}x_{i_2}x_{t},x_{i_3}],x_{i_4}],\cdots,x_{i_n}]),$$
where $t\leq i_{3}$;

For $[[\cdots[x_{i_1}x_{i_2}x_{i_3},x_{i_4}],\cdots],x_{i_n}]\circ x_t,$ we define
$$[[\cdots[x_{i_1}x_{i_2}x_{i_3},x_{i_4}],\cdots],x_{i_n}]\circ x_t=0.$$

Up to degree 4, we define multiplication in $A\<X\>$ that is consistent with identities (\ref{tp}), (\ref{11}), (\ref{12}) and (\ref{13}).
By straightforward calculation, we can check that an algebra $A\<X\>$ satisfies the defining identities of algebra $\M\TP\<X\>$. It remains to note that $A\<X\>\cong \M\TP\<X\>$.

\end{proof}

\section{Metabelian F-manifold operad}

A $F$-$manifold$ algebra is a triple $(\cdot,[\cdot,\cdot],X)$, where the multiplication $\cdot$ is associa\-tive-commutative and $[\cdot,\cdot]$ is a Lie bracket with the additional identity
\begin{multline}\label{eq:manifold}
    [a\cdot b,c \cdot d] = [a \cdot b,c] \cdot d +[a \cdot b,d] \cdot c + a \cdot [b,c \cdot d]+ b \cdot [a,c \cdot d]- \\
(a \cdot c) \cdot [b,d]-(b \cdot c) \cdot [a,d]-(b \cdot d) \cdot [a,c]-(a \cdot d) \cdot [b,c].
\end{multline}
A metabelian $F$-manifold algebra has additional identities (\ref{10}).
In metabelian $F$-manifold algebra, the identity (\ref{eq:manifold}) can be written as follows:
\begin{equation}\label{001}
[a \cdot b,c] \cdot d +[a \cdot b,d] \cdot c + a \cdot [b,c \cdot d]+ b \cdot [a,c \cdot d]=0.
\end{equation}

Since both multiplications are the same with transposed Poisson algebra, the identities (\ref{11}) and (\ref{12}) also hold in metabelian $F$-manifold algebra. 
In addition, one can obtain that
\begin{equation}\label{21}
[[a,b]c,d]e=-[[a,b]c,e]d
\end{equation}
and
\begin{equation}\label{22}
[abc,d]e=-[abc,e]d
\end{equation}
hold in metabelian $F$-manifold algebra.

Let us construct a basis of the metabelian $F$-manifold operad in terms of binary trees with two types of vertices $\bullet$ and $\circ$. We consider only trees of the following form:

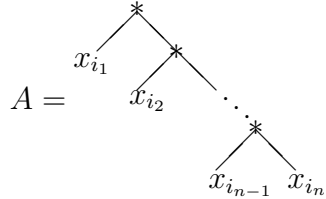
\begin{picture}(30,80)
\put(102,33){$A=$}
\put(147,68){$*$}
\put(150,70){\line(-1,-1){15}}
\put(150,70){\line(1,-1){15}}
\put(162,52){$*$}
\put(165,55){\line(-1,-1){15}}
\put(165,55){\line(1,-1){15}}
\put(182,33){\normalsize\rotatebox[origin=c]{320}{$\ldots$}}
\put(192,23){$*$}
\put(195,25){\line(-1,-1){15}}
\put(195,25){\line(1,-1){15}}
\put(126,49){$x_{i_1}$}
\put(147,35){$x_{i_2}$}
\put(177,3){$x_{i_{n-1}}$}
\put(207,3){$x_{i_{n}}$}
\end{picture}

\noindent We place on vertices of the tree $\bullet$ and $\circ$ in all possible ways. On the leaves of the tree, we place generators from the countable set $X$. This tree in a unique way corresponds to the sequence $(*_{x_{i_1}},*_{x_{i_2}},\ldots,*_{x_{i_{n-1}}},x_{i_{n}}).$

\begin{example}
If

\begin{picture}(30,80)
\put(100,60){$A=$}
\put(147,68){$\bullet$}
\put(150,70){\line(-1,-1){15}}
\put(150,70){\line(1,-1){15}}
\put(162,52){$\circ$}
\put(165,55){\line(-1,-1){15}}
\put(165,55){\line(1,-1){15}}
\put(126,49){$x_1$}
\put(147,35){$x_2$}
\put(177,35){$x_3$}
\end{picture}

\vspace*{-\baselineskip}
\vspace*{-\baselineskip}

\noindent then $A$ corresponds to $(\bullet_{x_1},\circ_{x_2},x_3)$. We define the conditions on a set of sequences as follows:
\end{example}

1) The sequence of several consecutive white dots are not allowed, i.e. we do not consider the sequences $(\ldots,\circ_{i_{k-1}},\circ_{i_k},\bullet_{i_{k+1}}\ldots)$ or $(\ldots,\circ_{i_{k-1}},\circ_{i_k},\circ_{i_{k+1}}\ldots)$. 

2) If there are several consecutive black dots in a sequence, then all corresponding generators of these vertices are ordered, i.e., for
$(\ldots,\circ_{i_{k-1}},\bullet_{i_k},\bullet_{i_{k+1}},\ldots,\bullet_{i_{l-1}},\circ_{i_{l}},\ldots)$, we have $i_k\geq i_{k+1}\geq\ldots\geq i_{l-1}$;

3) If the rightmost vertex is white then the rightmost generator is less than the previous one, i.e., for $(\ldots,\circ_{i_{n-1}},x_{i_n})$, we have $i_{n-1}\leq i_n$;

4) If a given consecutive sequence of black dots continues to the rightmost vertex and the number of black dots is bigger than $2$, then all the generators of these vertices are ordered and the rightmost generator is bigger than the previous one, i.e., for
$(\ldots,\bullet_{i_k},\bullet_{i_{k+1}},\ldots,\bullet_{i_{n-1}},x_{i_n})$, we have $i_k\geq i_{k+1}\geq\ldots\geq i_{n-1}<i_n$;

5) In condition $3$, if the number of black dots is not bigger than $2$, then the generators are ordered as in Lyndon-Shirshov words, i.e., the basis monomials of the free Lie algebra of degrees 2 and 3;

6) If we have $(\ldots,\circ_{i_{k-1}},\bullet_{i_{k}},\circ_{i_{k+1}},\ldots)$ and $\circ_{i_{k+1}}$ is not rightmost vertex, then $i_{k-1}\geq i_{k}$.

7) If we have $(\ldots,\circ_{i_{n-3}},\bullet_{i_{n-2}},\circ_{i_{n-1}},x_{i_n})$, then $i_{n-3}>i_{n-2}>i_{n}$ is not allowed.

For every such tree, we define a monomial from the metabelian $F$-manifold operad as follows:
the tree with $n$ leaves is a right-normed monomial of degree $n$, i.e., this is the monomial $x_{i_1}*(x_{i_2}*(\ldots(x_{i_{n-1}}*x_{i_n})\ldots))$. The black multiplication $\bullet$ corresponds to the Lie bracket $[\cdot,\cdot]$ and white multiplication $\circ$ corresponds to $\cdot$. We denote by $\mathcal{T}$ a set of trees that satisfy conditions (1), (2), (3), (4), (5), (6), (7), and we denote by $\mathcal{F}$ the set of monomials which correspond to the trees from $\mathcal{T}$.

Let us construct the trees from the set $\mathcal{T}$ for degrees $4$, $5$ and $6$. For degrees up to $3$, the construction is obvious. Firstly, we construct the tress which satisfies the conditions $(1)$-$(6)$. For degree $4$, these trees are

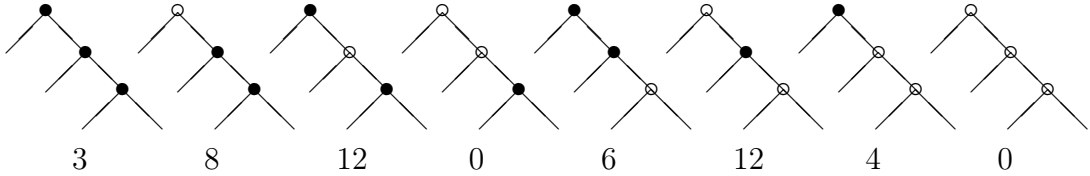
\begin{picture}(30,80)
\put(7,68){$\bullet$}
\put(10,70){\line(-1,-1){15}}
\put(10,70){\line(1,-1){15}}
\put(22,52){$\bullet$}
\put(25,55){\line(-1,-1){15}}
\put(25,55){\line(1,-1){15}}
\put(36,38){$\bullet$}
\put(39,41){\line(-1,-1){15}}
\put(39,41){\line(1,-1){15}}
\put(20,11){3}

\put(57,68){$\circ$}
\put(60,70){\line(-1,-1){15}}
\put(60,70){\line(1,-1){15}}
\put(72,52){$\bullet$}
\put(75,55){\line(-1,-1){15}}
\put(75,55){\line(1,-1){15}}
\put(86,38){$\bullet$}
\put(89,41){\line(-1,-1){15}}
\put(89,41){\line(1,-1){15}}
\put(70,11){8}

\put(107,68){$\bullet$}
\put(110,70){\line(-1,-1){15}}
\put(110,70){\line(1,-1){15}}
\put(122,52){$\circ$}
\put(125,55){\line(-1,-1){15}}
\put(125,55){\line(1,-1){15}}
\put(136,38){$\bullet$}
\put(139,41){\line(-1,-1){15}}
\put(139,41){\line(1,-1){15}}
\put(120,11){12}

\put(157,68){$\circ$}
\put(160,70){\line(-1,-1){15}}
\put(160,70){\line(1,-1){15}}
\put(172,52){$\circ$}
\put(175,55){\line(-1,-1){15}}
\put(175,55){\line(1,-1){15}}
\put(186,38){$\bullet$}
\put(189,41){\line(-1,-1){15}}
\put(189,41){\line(1,-1){15}}
\put(170,11){0}

\put(207,68){$\bullet$}
\put(210,70){\line(-1,-1){15}}
\put(210,70){\line(1,-1){15}}
\put(222,52){$\bullet$}
\put(225,55){\line(-1,-1){15}}
\put(225,55){\line(1,-1){15}}
\put(236,38){$\circ$}
\put(239,41){\line(-1,-1){15}}
\put(239,41){\line(1,-1){15}}
\put(220,11){6}

\put(257,68){$\circ$}
\put(260,70){\line(-1,-1){15}}
\put(260,70){\line(1,-1){15}}
\put(272,52){$\bullet$}
\put(275,55){\line(-1,-1){15}}
\put(275,55){\line(1,-1){15}}
\put(286,38){$\circ$}
\put(289,41){\line(-1,-1){15}}
\put(289,41){\line(1,-1){15}}
\put(270,11){12}

\put(307,68){$\bullet$}
\put(310,70){\line(-1,-1){15}}
\put(310,70){\line(1,-1){15}}
\put(322,52){$\circ$}
\put(325,55){\line(-1,-1){15}}
\put(325,55){\line(1,-1){15}}
\put(336,38){$\circ$}
\put(339,41){\line(-1,-1){15}}
\put(339,41){\line(1,-1){15}}
\put(320,11){4}

\put(357,68){$\circ$}
\put(360,70){\line(-1,-1){15}}
\put(360,70){\line(1,-1){15}}
\put(372,52){$\circ$}
\put(375,55){\line(-1,-1){15}}
\put(375,55){\line(1,-1){15}}
\put(386,38){$\circ$}
\put(389,41){\line(-1,-1){15}}
\put(389,41){\line(1,-1){15}}
\put(370,11){0}

\end{picture}

\noindent Under each tree we write the number of monomials that satisfy conditions $(1)$-$(6)$ for the corresponding tree. By condition $(7)$, we reduce $3$ monomials from the tree  $(\circ_{i_{1}},\bullet_{i_{2}},\circ_{i_{3}},x_{i_4})$. Finally, we obtain $42$ basis monomials.

Analogically, we construct the trees of degree $5$ which satisfies the conditions $(1)$-$(6)$. These trees are

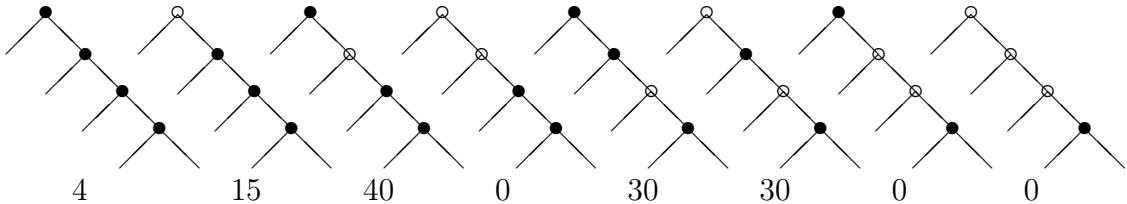
\begin{picture}(30,80)
\put(7,68){$\bullet$}
\put(10,70){\line(-1,-1){15}}
\put(10,70){\line(1,-1){15}}
\put(22,52){$\bullet$}
\put(25,55){\line(-1,-1){15}}
\put(25,55){\line(1,-1){15}}
\put(36,38){$\bullet$}
\put(39,41){\line(-1,-1){15}}
\put(39,41){\line(1,-1){15}}
\put(50,24){$\bullet$}
\put(53,27){\line(-1,-1){15}}
\put(53,27){\line(1,-1){15}}
\put(20,0){4}

\put(57,68){$\circ$}
\put(60,70){\line(-1,-1){15}}
\put(60,70){\line(1,-1){15}}
\put(72,52){$\bullet$}
\put(75,55){\line(-1,-1){15}}
\put(75,55){\line(1,-1){15}}
\put(86,38){$\bullet$}
\put(89,41){\line(-1,-1){15}}
\put(89,41){\line(1,-1){15}}
\put(100,24){$\bullet$}
\put(103,27){\line(-1,-1){15}}
\put(103,27){\line(1,-1){15}}
\put(80,0){15}

\put(107,68){$\bullet$}
\put(110,70){\line(-1,-1){15}}
\put(110,70){\line(1,-1){15}}
\put(122,52){$\circ$}
\put(125,55){\line(-1,-1){15}}
\put(125,55){\line(1,-1){15}}
\put(136,38){$\bullet$}
\put(139,41){\line(-1,-1){15}}
\put(139,41){\line(1,-1){15}}
\put(150,24){$\bullet$}
\put(153,27){\line(-1,-1){15}}
\put(153,27){\line(1,-1){15}}
\put(130,0){40}

\put(157,68){$\circ$}
\put(160,70){\line(-1,-1){15}}
\put(160,70){\line(1,-1){15}}
\put(172,52){$\circ$}
\put(175,55){\line(-1,-1){15}}
\put(175,55){\line(1,-1){15}}
\put(186,38){$\bullet$}
\put(189,41){\line(-1,-1){15}}
\put(189,41){\line(1,-1){15}}
\put(200,24){$\bullet$}
\put(203,27){\line(-1,-1){15}}
\put(203,27){\line(1,-1){15}}
\put(180,0){0}

\put(207,68){$\bullet$}
\put(210,70){\line(-1,-1){15}}
\put(210,70){\line(1,-1){15}}
\put(222,52){$\bullet$}
\put(225,55){\line(-1,-1){15}}
\put(225,55){\line(1,-1){15}}
\put(236,38){$\circ$}
\put(239,41){\line(-1,-1){15}}
\put(239,41){\line(1,-1){15}}
\put(250,24){$\bullet$}
\put(253,27){\line(-1,-1){15}}
\put(253,27){\line(1,-1){15}}
\put(230,0){30}

\put(257,68){$\circ$}
\put(260,70){\line(-1,-1){15}}
\put(260,70){\line(1,-1){15}}
\put(272,52){$\bullet$}
\put(275,55){\line(-1,-1){15}}
\put(275,55){\line(1,-1){15}}
\put(286,38){$\circ$}
\put(289,41){\line(-1,-1){15}}
\put(289,41){\line(1,-1){15}}
\put(300,24){$\bullet$}
\put(303,27){\line(-1,-1){15}}
\put(303,27){\line(1,-1){15}}
\put(280,0){30}

\put(307,68){$\bullet$}
\put(310,70){\line(-1,-1){15}}
\put(310,70){\line(1,-1){15}}
\put(322,52){$\circ$}
\put(325,55){\line(-1,-1){15}}
\put(325,55){\line(1,-1){15}}
\put(336,38){$\circ$}
\put(339,41){\line(-1,-1){15}}
\put(339,41){\line(1,-1){15}}
\put(350,24){$\bullet$}
\put(353,27){\line(-1,-1){15}}
\put(353,27){\line(1,-1){15}}
\put(330,0){0}

\put(357,68){$\circ$}
\put(360,70){\line(-1,-1){15}}
\put(360,70){\line(1,-1){15}}
\put(372,52){$\circ$}
\put(375,55){\line(-1,-1){15}}
\put(375,55){\line(1,-1){15}}
\put(386,38){$\circ$}
\put(389,41){\line(-1,-1){15}}
\put(389,41){\line(1,-1){15}}
\put(400,24){$\bullet$}
\put(403,27){\line(-1,-1){15}}
\put(403,27){\line(1,-1){15}}
\put(380,0){0}

\end{picture}

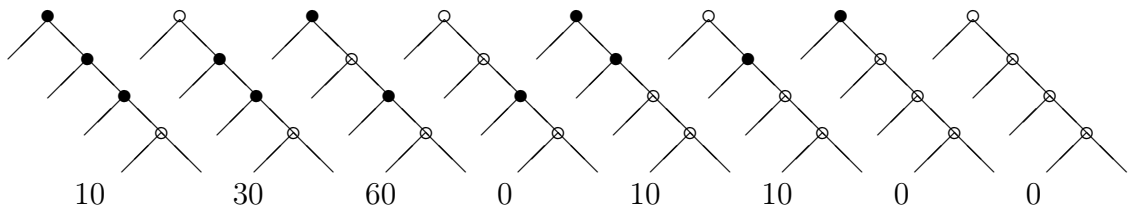
\begin{picture}(30,80)
\put(7,68){$\bullet$}
\put(10,70){\line(-1,-1){15}}
\put(10,70){\line(1,-1){15}}
\put(22,52){$\bullet$}
\put(25,55){\line(-1,-1){15}}
\put(25,55){\line(1,-1){15}}
\put(36,38){$\bullet$}
\put(39,41){\line(-1,-1){15}}
\put(39,41){\line(1,-1){15}}
\put(50,24){$\circ$}
\put(53,27){\line(-1,-1){15}}
\put(53,27){\line(1,-1){15}}
\put(20,0){10}

\put(57,68){$\circ$}
\put(60,70){\line(-1,-1){15}}
\put(60,70){\line(1,-1){15}}
\put(72,52){$\bullet$}
\put(75,55){\line(-1,-1){15}}
\put(75,55){\line(1,-1){15}}
\put(86,38){$\bullet$}
\put(89,41){\line(-1,-1){15}}
\put(89,41){\line(1,-1){15}}
\put(100,24){$\circ$}
\put(103,27){\line(-1,-1){15}}
\put(103,27){\line(1,-1){15}}
\put(80,0){30}

\put(107,68){$\bullet$}
\put(110,70){\line(-1,-1){15}}
\put(110,70){\line(1,-1){15}}
\put(122,52){$\circ$}
\put(125,55){\line(-1,-1){15}}
\put(125,55){\line(1,-1){15}}
\put(136,38){$\bullet$}
\put(139,41){\line(-1,-1){15}}
\put(139,41){\line(1,-1){15}}
\put(150,24){$\circ$}
\put(153,27){\line(-1,-1){15}}
\put(153,27){\line(1,-1){15}}
\put(130,0){60}

\put(157,68){$\circ$}
\put(160,70){\line(-1,-1){15}}
\put(160,70){\line(1,-1){15}}
\put(172,52){$\circ$}
\put(175,55){\line(-1,-1){15}}
\put(175,55){\line(1,-1){15}}
\put(186,38){$\bullet$}
\put(189,41){\line(-1,-1){15}}
\put(189,41){\line(1,-1){15}}
\put(200,24){$\circ$}
\put(203,27){\line(-1,-1){15}}
\put(203,27){\line(1,-1){15}}
\put(180,0){0}

\put(207,68){$\bullet$}
\put(210,70){\line(-1,-1){15}}
\put(210,70){\line(1,-1){15}}
\put(222,52){$\bullet$}
\put(225,55){\line(-1,-1){15}}
\put(225,55){\line(1,-1){15}}
\put(236,38){$\circ$}
\put(239,41){\line(-1,-1){15}}
\put(239,41){\line(1,-1){15}}
\put(250,24){$\circ$}
\put(253,27){\line(-1,-1){15}}
\put(253,27){\line(1,-1){15}}
\put(230,0){10}

\put(257,68){$\circ$}
\put(260,70){\line(-1,-1){15}}
\put(260,70){\line(1,-1){15}}
\put(272,52){$\bullet$}
\put(275,55){\line(-1,-1){15}}
\put(275,55){\line(1,-1){15}}
\put(286,38){$\circ$}
\put(289,41){\line(-1,-1){15}}
\put(289,41){\line(1,-1){15}}
\put(300,24){$\circ$}
\put(303,27){\line(-1,-1){15}}
\put(303,27){\line(1,-1){15}}
\put(280,0){10}

\put(307,68){$\bullet$}
\put(310,70){\line(-1,-1){15}}
\put(310,70){\line(1,-1){15}}
\put(322,52){$\circ$}
\put(325,55){\line(-1,-1){15}}
\put(325,55){\line(1,-1){15}}
\put(336,38){$\circ$}
\put(339,41){\line(-1,-1){15}}
\put(339,41){\line(1,-1){15}}
\put(350,24){$\circ$}
\put(353,27){\line(-1,-1){15}}
\put(353,27){\line(1,-1){15}}
\put(330,0){0}

\put(357,68){$\circ$}
\put(360,70){\line(-1,-1){15}}
\put(360,70){\line(1,-1){15}}
\put(372,52){$\circ$}
\put(375,55){\line(-1,-1){15}}
\put(375,55){\line(1,-1){15}}
\put(386,38){$\circ$}
\put(389,41){\line(-1,-1){15}}
\put(389,41){\line(1,-1){15}}
\put(400,24){$\circ$}
\put(403,27){\line(-1,-1){15}}
\put(403,27){\line(1,-1){15}}
\put(380,0){0}

\end{picture}

\noindent The numbers under each tree are written in the same condition as before. By condition $(7)$ we reduce $15$ monomials from the tree  $(\bullet_{i_{1}},\circ_{i_{2}},\bullet_{i_{3}},\circ_{i_{4}},x_{i_5})$. Finally, we obtain $224$ basis monomials.

For degree $6$, we have

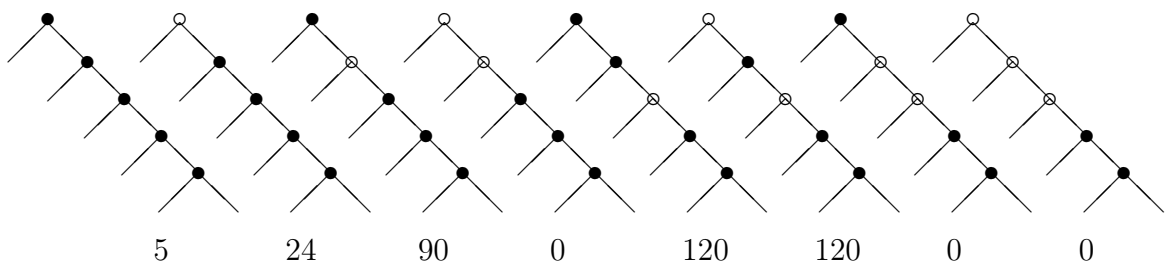
\begin{picture}(30,80)
\put(7,68){$\bullet$}
\put(10,70){\line(-1,-1){15}}
\put(10,70){\line(1,-1){15}}
\put(22,52){$\bullet$}
\put(25,55){\line(-1,-1){15}}
\put(25,55){\line(1,-1){15}}
\put(36,38){$\bullet$}
\put(39,41){\line(-1,-1){15}}
\put(39,41){\line(1,-1){15}}
\put(50,24){$\bullet$}
\put(53,27){\line(-1,-1){15}}
\put(53,27){\line(1,-1){15}}
\put(64,10){$\bullet$}
\put(67,13){\line(-1,-1){15}}
\put(67,13){\line(1,-1){15}}
\put(50,-20){5}

\put(57,68){$\circ$}
\put(60,70){\line(-1,-1){15}}
\put(60,70){\line(1,-1){15}}
\put(72,52){$\bullet$}
\put(75,55){\line(-1,-1){15}}
\put(75,55){\line(1,-1){15}}
\put(86,38){$\bullet$}
\put(89,41){\line(-1,-1){15}}
\put(89,41){\line(1,-1){15}}
\put(100,24){$\bullet$}
\put(103,27){\line(-1,-1){15}}
\put(103,27){\line(1,-1){15}}
\put(114,10){$\bullet$}
\put(117,13){\line(-1,-1){15}}
\put(117,13){\line(1,-1){15}}
\put(100,-20){24}

\put(107,68){$\bullet$}
\put(110,70){\line(-1,-1){15}}
\put(110,70){\line(1,-1){15}}
\put(122,52){$\circ$}
\put(125,55){\line(-1,-1){15}}
\put(125,55){\line(1,-1){15}}
\put(136,38){$\bullet$}
\put(139,41){\line(-1,-1){15}}
\put(139,41){\line(1,-1){15}}
\put(150,24){$\bullet$}
\put(153,27){\line(-1,-1){15}}
\put(153,27){\line(1,-1){15}}
\put(164,10){$\bullet$}
\put(167,13){\line(-1,-1){15}}
\put(167,13){\line(1,-1){15}}
\put(150,-20){90}

\put(157,68){$\circ$}
\put(160,70){\line(-1,-1){15}}
\put(160,70){\line(1,-1){15}}
\put(172,52){$\circ$}
\put(175,55){\line(-1,-1){15}}
\put(175,55){\line(1,-1){15}}
\put(186,38){$\bullet$}
\put(189,41){\line(-1,-1){15}}
\put(189,41){\line(1,-1){15}}
\put(200,24){$\bullet$}
\put(203,27){\line(-1,-1){15}}
\put(203,27){\line(1,-1){15}}
\put(214,10){$\bullet$}
\put(217,13){\line(-1,-1){15}}
\put(217,13){\line(1,-1){15}}
\put(200,-20){0}

\put(207,68){$\bullet$}
\put(210,70){\line(-1,-1){15}}
\put(210,70){\line(1,-1){15}}
\put(222,52){$\bullet$}
\put(225,55){\line(-1,-1){15}}
\put(225,55){\line(1,-1){15}}
\put(236,38){$\circ$}
\put(239,41){\line(-1,-1){15}}
\put(239,41){\line(1,-1){15}}
\put(250,24){$\bullet$}
\put(253,27){\line(-1,-1){15}}
\put(253,27){\line(1,-1){15}}
\put(264,10){$\bullet$}
\put(267,13){\line(-1,-1){15}}
\put(267,13){\line(1,-1){15}}
\put(250,-20){120}

\put(257,68){$\circ$}
\put(260,70){\line(-1,-1){15}}
\put(260,70){\line(1,-1){15}}
\put(272,52){$\bullet$}
\put(275,55){\line(-1,-1){15}}
\put(275,55){\line(1,-1){15}}
\put(286,38){$\circ$}
\put(289,41){\line(-1,-1){15}}
\put(289,41){\line(1,-1){15}}
\put(300,24){$\bullet$}
\put(303,27){\line(-1,-1){15}}
\put(303,27){\line(1,-1){15}}
\put(314,10){$\bullet$}
\put(317,13){\line(-1,-1){15}}
\put(317,13){\line(1,-1){15}}
\put(300,-20){120}

\put(307,68){$\bullet$}
\put(310,70){\line(-1,-1){15}}
\put(310,70){\line(1,-1){15}}
\put(322,52){$\circ$}
\put(325,55){\line(-1,-1){15}}
\put(325,55){\line(1,-1){15}}
\put(336,38){$\circ$}
\put(339,41){\line(-1,-1){15}}
\put(339,41){\line(1,-1){15}}
\put(350,24){$\bullet$}
\put(353,27){\line(-1,-1){15}}
\put(353,27){\line(1,-1){15}}
\put(364,10){$\bullet$}
\put(367,13){\line(-1,-1){15}}
\put(367,13){\line(1,-1){15}}
\put(350,-20){0}

\put(357,68){$\circ$}
\put(360,70){\line(-1,-1){15}}
\put(360,70){\line(1,-1){15}}
\put(372,52){$\circ$}
\put(375,55){\line(-1,-1){15}}
\put(375,55){\line(1,-1){15}}
\put(386,38){$\circ$}
\put(389,41){\line(-1,-1){15}}
\put(389,41){\line(1,-1){15}}
\put(400,24){$\bullet$}
\put(403,27){\line(-1,-1){15}}
\put(403,27){\line(1,-1){15}}
\put(414,10){$\bullet$}
\put(417,13){\line(-1,-1){15}}
\put(417,13){\line(1,-1){15}}
\put(400,-20){0}

\end{picture}

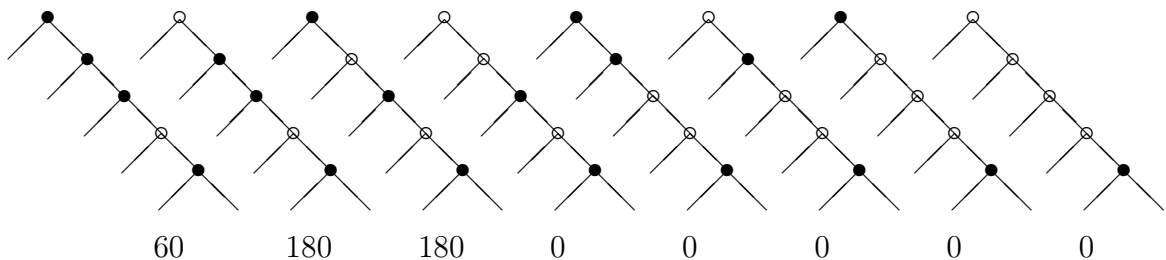
$$ $$

\begin{picture}(30,80)
\put(7,68){$\bullet$}
\put(10,70){\line(-1,-1){15}}
\put(10,70){\line(1,-1){15}}
\put(22,52){$\bullet$}
\put(25,55){\line(-1,-1){15}}
\put(25,55){\line(1,-1){15}}
\put(36,38){$\bullet$}
\put(39,41){\line(-1,-1){15}}
\put(39,41){\line(1,-1){15}}
\put(50,24){$\circ$}
\put(53,27){\line(-1,-1){15}}
\put(53,27){\line(1,-1){15}}
\put(64,10){$\bullet$}
\put(67,13){\line(-1,-1){15}}
\put(67,13){\line(1,-1){15}}
\put(50,-20){60}

\put(57,68){$\circ$}
\put(60,70){\line(-1,-1){15}}
\put(60,70){\line(1,-1){15}}
\put(72,52){$\bullet$}
\put(75,55){\line(-1,-1){15}}
\put(75,55){\line(1,-1){15}}
\put(86,38){$\bullet$}
\put(89,41){\line(-1,-1){15}}
\put(89,41){\line(1,-1){15}}
\put(100,24){$\circ$}
\put(103,27){\line(-1,-1){15}}
\put(103,27){\line(1,-1){15}}
\put(114,10){$\bullet$}
\put(117,13){\line(-1,-1){15}}
\put(117,13){\line(1,-1){15}}
\put(100,-20){180}

\put(107,68){$\bullet$}
\put(110,70){\line(-1,-1){15}}
\put(110,70){\line(1,-1){15}}
\put(122,52){$\circ$}
\put(125,55){\line(-1,-1){15}}
\put(125,55){\line(1,-1){15}}
\put(136,38){$\bullet$}
\put(139,41){\line(-1,-1){15}}
\put(139,41){\line(1,-1){15}}
\put(150,24){$\circ$}
\put(153,27){\line(-1,-1){15}}
\put(153,27){\line(1,-1){15}}
\put(164,10){$\bullet$}
\put(167,13){\line(-1,-1){15}}
\put(167,13){\line(1,-1){15}}
\put(150,-20){180}

\put(157,68){$\circ$}
\put(160,70){\line(-1,-1){15}}
\put(160,70){\line(1,-1){15}}
\put(172,52){$\circ$}
\put(175,55){\line(-1,-1){15}}
\put(175,55){\line(1,-1){15}}
\put(186,38){$\bullet$}
\put(189,41){\line(-1,-1){15}}
\put(189,41){\line(1,-1){15}}
\put(200,24){$\circ$}
\put(203,27){\line(-1,-1){15}}
\put(203,27){\line(1,-1){15}}
\put(214,10){$\bullet$}
\put(217,13){\line(-1,-1){15}}
\put(217,13){\line(1,-1){15}}
\put(200,-20){0}

\put(207,68){$\bullet$}
\put(210,70){\line(-1,-1){15}}
\put(210,70){\line(1,-1){15}}
\put(222,52){$\bullet$}
\put(225,55){\line(-1,-1){15}}
\put(225,55){\line(1,-1){15}}
\put(236,38){$\circ$}
\put(239,41){\line(-1,-1){15}}
\put(239,41){\line(1,-1){15}}
\put(250,24){$\circ$}
\put(253,27){\line(-1,-1){15}}
\put(253,27){\line(1,-1){15}}
\put(264,10){$\bullet$}
\put(267,13){\line(-1,-1){15}}
\put(267,13){\line(1,-1){15}}
\put(250,-20){0}

\put(257,68){$\circ$}
\put(260,70){\line(-1,-1){15}}
\put(260,70){\line(1,-1){15}}
\put(272,52){$\bullet$}
\put(275,55){\line(-1,-1){15}}
\put(275,55){\line(1,-1){15}}
\put(286,38){$\circ$}
\put(289,41){\line(-1,-1){15}}
\put(289,41){\line(1,-1){15}}
\put(300,24){$\circ$}
\put(303,27){\line(-1,-1){15}}
\put(303,27){\line(1,-1){15}}
\put(314,10){$\bullet$}
\put(317,13){\line(-1,-1){15}}
\put(317,13){\line(1,-1){15}}
\put(300,-20){0}

\put(307,68){$\bullet$}
\put(310,70){\line(-1,-1){15}}
\put(310,70){\line(1,-1){15}}
\put(322,52){$\circ$}
\put(325,55){\line(-1,-1){15}}
\put(325,55){\line(1,-1){15}}
\put(336,38){$\circ$}
\put(339,41){\line(-1,-1){15}}
\put(339,41){\line(1,-1){15}}
\put(350,24){$\circ$}
\put(353,27){\line(-1,-1){15}}
\put(353,27){\line(1,-1){15}}
\put(364,10){$\bullet$}
\put(367,13){\line(-1,-1){15}}
\put(367,13){\line(1,-1){15}}
\put(350,-20){0}

\put(357,68){$\circ$}
\put(360,70){\line(-1,-1){15}}
\put(360,70){\line(1,-1){15}}
\put(372,52){$\circ$}
\put(375,55){\line(-1,-1){15}}
\put(375,55){\line(1,-1){15}}
\put(386,38){$\circ$}
\put(389,41){\line(-1,-1){15}}
\put(389,41){\line(1,-1){15}}
\put(400,24){$\circ$}
\put(403,27){\line(-1,-1){15}}
\put(403,27){\line(1,-1){15}}
\put(414,10){$\bullet$}
\put(417,13){\line(-1,-1){15}}
\put(417,13){\line(1,-1){15}}
\put(400,-20){0}

\end{picture}

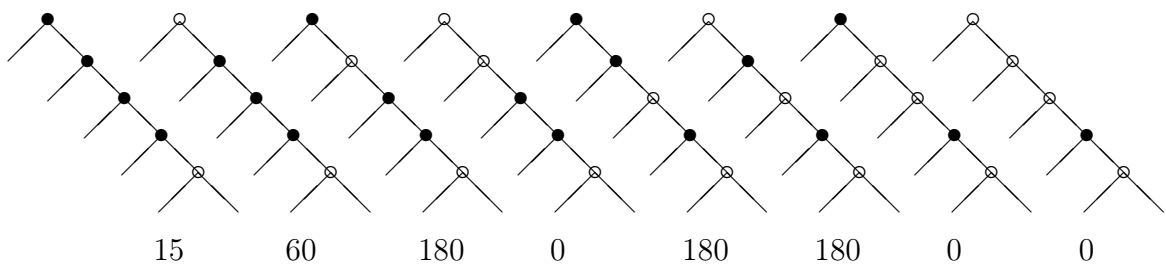
$$ $$

\begin{picture}(30,80)
\put(7,68){$\bullet$}
\put(10,70){\line(-1,-1){15}}
\put(10,70){\line(1,-1){15}}
\put(22,52){$\bullet$}
\put(25,55){\line(-1,-1){15}}
\put(25,55){\line(1,-1){15}}
\put(36,38){$\bullet$}
\put(39,41){\line(-1,-1){15}}
\put(39,41){\line(1,-1){15}}
\put(50,24){$\bullet$}
\put(53,27){\line(-1,-1){15}}
\put(53,27){\line(1,-1){15}}
\put(64,10){$\circ$}
\put(67,13){\line(-1,-1){15}}
\put(67,13){\line(1,-1){15}}
\put(50,-20){15}

\put(57,68){$\circ$}
\put(60,70){\line(-1,-1){15}}
\put(60,70){\line(1,-1){15}}
\put(72,52){$\bullet$}
\put(75,55){\line(-1,-1){15}}
\put(75,55){\line(1,-1){15}}
\put(86,38){$\bullet$}
\put(89,41){\line(-1,-1){15}}
\put(89,41){\line(1,-1){15}}
\put(100,24){$\bullet$}
\put(103,27){\line(-1,-1){15}}
\put(103,27){\line(1,-1){15}}
\put(114,10){$\circ$}
\put(117,13){\line(-1,-1){15}}
\put(117,13){\line(1,-1){15}}
\put(100,-20){60}

\put(107,68){$\bullet$}
\put(110,70){\line(-1,-1){15}}
\put(110,70){\line(1,-1){15}}
\put(122,52){$\circ$}
\put(125,55){\line(-1,-1){15}}
\put(125,55){\line(1,-1){15}}
\put(136,38){$\bullet$}
\put(139,41){\line(-1,-1){15}}
\put(139,41){\line(1,-1){15}}
\put(150,24){$\bullet$}
\put(153,27){\line(-1,-1){15}}
\put(153,27){\line(1,-1){15}}
\put(164,10){$\circ$}
\put(167,13){\line(-1,-1){15}}
\put(167,13){\line(1,-1){15}}
\put(150,-20){180}

\put(157,68){$\circ$}
\put(160,70){\line(-1,-1){15}}
\put(160,70){\line(1,-1){15}}
\put(172,52){$\circ$}
\put(175,55){\line(-1,-1){15}}
\put(175,55){\line(1,-1){15}}
\put(186,38){$\bullet$}
\put(189,41){\line(-1,-1){15}}
\put(189,41){\line(1,-1){15}}
\put(200,24){$\bullet$}
\put(203,27){\line(-1,-1){15}}
\put(203,27){\line(1,-1){15}}
\put(214,10){$\circ$}
\put(217,13){\line(-1,-1){15}}
\put(217,13){\line(1,-1){15}}
\put(200,-20){0}

\put(207,68){$\bullet$}
\put(210,70){\line(-1,-1){15}}
\put(210,70){\line(1,-1){15}}
\put(222,52){$\bullet$}
\put(225,55){\line(-1,-1){15}}
\put(225,55){\line(1,-1){15}}
\put(236,38){$\circ$}
\put(239,41){\line(-1,-1){15}}
\put(239,41){\line(1,-1){15}}
\put(250,24){$\bullet$}
\put(253,27){\line(-1,-1){15}}
\put(253,27){\line(1,-1){15}}
\put(264,10){$\circ$}
\put(267,13){\line(-1,-1){15}}
\put(267,13){\line(1,-1){15}}
\put(250,-20){180}

\put(257,68){$\circ$}
\put(260,70){\line(-1,-1){15}}
\put(260,70){\line(1,-1){15}}
\put(272,52){$\bullet$}
\put(275,55){\line(-1,-1){15}}
\put(275,55){\line(1,-1){15}}
\put(286,38){$\circ$}
\put(289,41){\line(-1,-1){15}}
\put(289,41){\line(1,-1){15}}
\put(300,24){$\bullet$}
\put(303,27){\line(-1,-1){15}}
\put(303,27){\line(1,-1){15}}
\put(314,10){$\circ$}
\put(317,13){\line(-1,-1){15}}
\put(317,13){\line(1,-1){15}}
\put(300,-20){180}

\put(307,68){$\bullet$}
\put(310,70){\line(-1,-1){15}}
\put(310,70){\line(1,-1){15}}
\put(322,52){$\circ$}
\put(325,55){\line(-1,-1){15}}
\put(325,55){\line(1,-1){15}}
\put(336,38){$\circ$}
\put(339,41){\line(-1,-1){15}}
\put(339,41){\line(1,-1){15}}
\put(350,24){$\bullet$}
\put(353,27){\line(-1,-1){15}}
\put(353,27){\line(1,-1){15}}
\put(364,10){$\circ$}
\put(367,13){\line(-1,-1){15}}
\put(367,13){\line(1,-1){15}}
\put(350,-20){0}

\put(357,68){$\circ$}
\put(360,70){\line(-1,-1){15}}
\put(360,70){\line(1,-1){15}}
\put(372,52){$\circ$}
\put(375,55){\line(-1,-1){15}}
\put(375,55){\line(1,-1){15}}
\put(386,38){$\circ$}
\put(389,41){\line(-1,-1){15}}
\put(389,41){\line(1,-1){15}}
\put(400,24){$\bullet$}
\put(403,27){\line(-1,-1){15}}
\put(403,27){\line(1,-1){15}}
\put(414,10){$\circ$}
\put(417,13){\line(-1,-1){15}}
\put(417,13){\line(1,-1){15}}
\put(400,-20){0}

\end{picture}

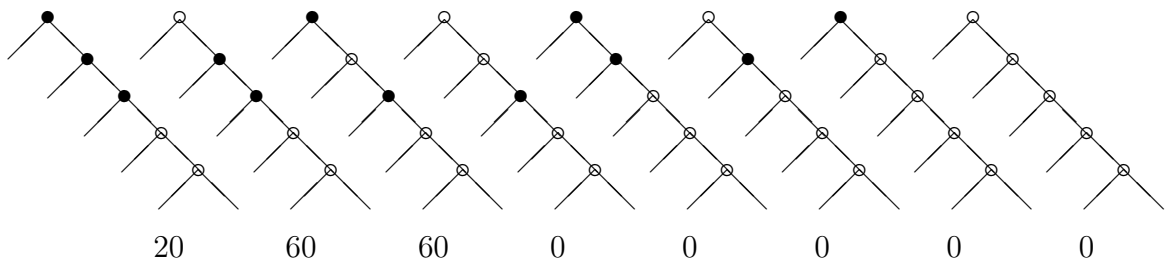
$$ $$

\begin{picture}(30,80)
\put(7,68){$\bullet$}
\put(10,70){\line(-1,-1){15}}
\put(10,70){\line(1,-1){15}}
\put(22,52){$\bullet$}
\put(25,55){\line(-1,-1){15}}
\put(25,55){\line(1,-1){15}}
\put(36,38){$\bullet$}
\put(39,41){\line(-1,-1){15}}
\put(39,41){\line(1,-1){15}}
\put(50,24){$\circ$}
\put(53,27){\line(-1,-1){15}}
\put(53,27){\line(1,-1){15}}
\put(64,10){$\circ$}
\put(67,13){\line(-1,-1){15}}
\put(67,13){\line(1,-1){15}}
\put(50,-20){20}

\put(57,68){$\circ$}
\put(60,70){\line(-1,-1){15}}
\put(60,70){\line(1,-1){15}}
\put(72,52){$\bullet$}
\put(75,55){\line(-1,-1){15}}
\put(75,55){\line(1,-1){15}}
\put(86,38){$\bullet$}
\put(89,41){\line(-1,-1){15}}
\put(89,41){\line(1,-1){15}}
\put(100,24){$\circ$}
\put(103,27){\line(-1,-1){15}}
\put(103,27){\line(1,-1){15}}
\put(114,10){$\circ$}
\put(117,13){\line(-1,-1){15}}
\put(117,13){\line(1,-1){15}}
\put(100,-20){60}

\put(107,68){$\bullet$}
\put(110,70){\line(-1,-1){15}}
\put(110,70){\line(1,-1){15}}
\put(122,52){$\circ$}
\put(125,55){\line(-1,-1){15}}
\put(125,55){\line(1,-1){15}}
\put(136,38){$\bullet$}
\put(139,41){\line(-1,-1){15}}
\put(139,41){\line(1,-1){15}}
\put(150,24){$\circ$}
\put(153,27){\line(-1,-1){15}}
\put(153,27){\line(1,-1){15}}
\put(164,10){$\circ$}
\put(167,13){\line(-1,-1){15}}
\put(167,13){\line(1,-1){15}}
\put(150,-20){60}

\put(157,68){$\circ$}
\put(160,70){\line(-1,-1){15}}
\put(160,70){\line(1,-1){15}}
\put(172,52){$\circ$}
\put(175,55){\line(-1,-1){15}}
\put(175,55){\line(1,-1){15}}
\put(186,38){$\bullet$}
\put(189,41){\line(-1,-1){15}}
\put(189,41){\line(1,-1){15}}
\put(200,24){$\circ$}
\put(203,27){\line(-1,-1){15}}
\put(203,27){\line(1,-1){15}}
\put(214,10){$\circ$}
\put(217,13){\line(-1,-1){15}}
\put(217,13){\line(1,-1){15}}
\put(200,-20){0}

\put(207,68){$\bullet$}
\put(210,70){\line(-1,-1){15}}
\put(210,70){\line(1,-1){15}}
\put(222,52){$\bullet$}
\put(225,55){\line(-1,-1){15}}
\put(225,55){\line(1,-1){15}}
\put(236,38){$\circ$}
\put(239,41){\line(-1,-1){15}}
\put(239,41){\line(1,-1){15}}
\put(250,24){$\circ$}
\put(253,27){\line(-1,-1){15}}
\put(253,27){\line(1,-1){15}}
\put(264,10){$\circ$}
\put(267,13){\line(-1,-1){15}}
\put(267,13){\line(1,-1){15}}
\put(250,-20){0}

\put(257,68){$\circ$}
\put(260,70){\line(-1,-1){15}}
\put(260,70){\line(1,-1){15}}
\put(272,52){$\bullet$}
\put(275,55){\line(-1,-1){15}}
\put(275,55){\line(1,-1){15}}
\put(286,38){$\circ$}
\put(289,41){\line(-1,-1){15}}
\put(289,41){\line(1,-1){15}}
\put(300,24){$\circ$}
\put(303,27){\line(-1,-1){15}}
\put(303,27){\line(1,-1){15}}
\put(314,10){$\circ$}
\put(317,13){\line(-1,-1){15}}
\put(317,13){\line(1,-1){15}}
\put(300,-20){0}

\put(307,68){$\bullet$}
\put(310,70){\line(-1,-1){15}}
\put(310,70){\line(1,-1){15}}
\put(322,52){$\circ$}
\put(325,55){\line(-1,-1){15}}
\put(325,55){\line(1,-1){15}}
\put(336,38){$\circ$}
\put(339,41){\line(-1,-1){15}}
\put(339,41){\line(1,-1){15}}
\put(350,24){$\circ$}
\put(353,27){\line(-1,-1){15}}
\put(353,27){\line(1,-1){15}}
\put(364,10){$\circ$}
\put(367,13){\line(-1,-1){15}}
\put(367,13){\line(1,-1){15}}
\put(350,-20){0}

\put(357,68){$\circ$}
\put(360,70){\line(-1,-1){15}}
\put(360,70){\line(1,-1){15}}
\put(372,52){$\circ$}
\put(375,55){\line(-1,-1){15}}
\put(375,55){\line(1,-1){15}}
\put(386,38){$\circ$}
\put(389,41){\line(-1,-1){15}}
\put(389,41){\line(1,-1){15}}
\put(400,24){$\circ$}
\put(403,27){\line(-1,-1){15}}
\put(403,27){\line(1,-1){15}}
\put(414,10){$\circ$}
\put(417,13){\line(-1,-1){15}}
\put(417,13){\line(1,-1){15}}
\put(400,-20){0}

\end{picture}

$$ $$

\noindent By condition $(7)$, we reduce $45$ monomials from the tree  $(\bullet_{i_{1}},\bullet_{i_{2}},\circ_{i_{3}},\bullet_{i_{4}},\circ_{i_{5}},x_{i_6})$ and $45$ monomials from the tree $(\circ_{i_{1}},\bullet_{i_{2}},\circ_{i_{3}},\bullet_{i_{4}},\circ_{i_{5}},x_{i_6})$. Finally, we obtain $1444$ basis monomials. For degree $7$, analogical constriction of trees gives $10870$ monomials. Calculating the dimension of the operad $\M\F$-$\manifold$, one can obtain:
\begin{center}
\begin{tabular}{c|cccccccc}
 $n$ & 1 & 2 & 3 & 4 & 5 & 6 & 7 \\
 \hline 
 $\dim(\M\F\textrm{-}\manifold(n)) $ & 1 & 2 & 9 & 42 & 224 & 1444 & 10870
\end{tabular}
\end{center}

\begin{theorem}
The set $\mathcal{F}$ is a basis of the metabelian $\F\textrm{-}\manifold$ operad.
\end{theorem}
\begin{proof}
Firstly, we show that any monomial can be written as a sum of monomials which satisfy conditions (1)-(7). The condition (1) holds by associative and metabelian identities on $\cdot$. The condition (2) holds by identities (\ref{11}) and (\ref{12}). The condition (3) holds by commutative identity on $\cdot$. The conditions (4) and (5) can be achieved by identities of metabelian Lie algebra on $[\cdot,\cdot]$. The condition (6) can be achieved by identities (\ref{21}) and (\ref{22}). The condition (7) holds by identity (\ref{001}).

Finally, to obtain the spanning set $\mathcal{F}$ in operad $\M\F\textrm{-}\manifold$ we use the rewriting rules of operads $\Lie$ and $\Com$, and rewriting rules obtained from identities (\ref{10}), $(\ref{11})$, $(\ref{12})$, $(\ref{001})$, $(\ref{21})$, $(\ref{22})$.

To prove that the set $\mathcal{F}$ is a basis, we have to show that the compositions of rewriting rules are trivial, see \cite{bremner-dotsenko}. The calculations are straightforward or these long calculations can be done using the computer software \cite{DotsHij}.
\end{proof}


 \end{document}